\documentclass[11pt]{article}
\usepackage{a4wide}
\usepackage{amsmath,amsthm,amssymb}
\usepackage{units}
\usepackage{url}
\usepackage{calc, graphicx} 
\usepackage{bbm}
\usepackage{dsfont}
\newcommand{\ind}{\ensuremath{\mathds{1}}}

\newtheorem{thm}{Theorem}[section]
\newtheorem{lem}[thm]{Lemma}
\newtheorem{prop}[thm]{Proposition}

{\theoremstyle{definition}
  
}

\newcommand{\N}{\ensuremath{\mathbb N}}
\newcommand{\C}{\ensuremath{\mathbb{C}}}
\newcommand{\R}{\ensuremath{\mathbb{R}}}
\newcommand{\E}{\ensuremath{\mathbb{E}}}
\newcommand{\Prob}{\ensuremath{\mathbb{P}}}
\newcommand{\bo}{\ensuremath{\mathrm{O}}}
\newcommand{\so}{\ensuremath{\mathrm{o}}}

\newcommand{\Cov}{\ensuremath{\mathrm{Cov}}}

\newsavebox{\smlmat}
\savebox{\smlmat}{$\left[\begin{smallmatrix}1 & 2 \\ 2 &1\end{smallmatrix} \right]$ }

\title{Refined Asymptotics for the Composition of Cyclic Urns}
\author{Noela S.~M\"uller and Ralph Neininger\\
Institute for Mathematics\\
J.W.~Goethe University\\
60054 Frankfurt a.M.\\
Germany\\ \\
Email: {\tt \{nmueller,neiningr\}@math.uni-frankfurt.de}}
\begin{document}

\maketitle

\begin{abstract}
A cyclic urn is an urn model for balls of types $0,\ldots,m-1$. The urn starts at time zero with an initial configuration. Then, in each time step, first a ball is drawn from the urn uniformly and independently from the past. If its type is $j$, it is then returned to the urn together with a new ball of type $j+1 \mod m$. The case $m=2$ is the well-known Friedman urn. The composition vector, i.e., the vector of the numbers of balls of each type after $n$ steps is, after normalization, known to be asymptotically normal for $2\le m\le 6$. For $m\ge 7$ the normalized composition vector is known not to converge. However, there is an almost sure approximation by a periodic random vector.

In the present paper  the asymptotic fluctuations around this periodic random vector are identified.  We show that these fluctuations are asymptotically normal for all $7\le m\le 12$. For $m\ge 13$ we also find asymptotically normal fluctuations when normalizing in a more refined way. These fluctuations are of maximal dimension $m-1$ only when $6$ does not divide $m$. For $m$ being a multiple of $6$ the fluctuations are supported by a two-dimensional subspace.  \end{abstract}

\noindent
{\textbf{MSC2010:} 60F05, 60F15, 60C05, 60J10.


\section{Introduction and result}
A cyclic urn is an urn model with a fixed number $m\ge 2$ of possible colors of balls which we call types $0,\ldots,m-1$. We assume that initially there is one ball of type $0$. In each step, we draw a ball from the urn, uniformly from within the balls in the urn and independently of the history of the urn process. If its type is $j\in\{0,\ldots,m-1\}$ it is placed back to the urn together with a new ball of type $j+1 \mod m$. These steps are iterated.

We denote by $R_n=(R_{n,0},\ldots,R_{n,m-1})^t$ the (column) vector of the numbers of balls of each type after $n$ steps when starting with one ball of type $0$. Hence, we have $R_0=e_0$ where $e_j$ denotes the $j$-th unit vector in $\R^m$, indexing the unit vectors by $0,\ldots,m-1$. For fixed $m\ge 2$ we denote the $m$-th elementary root of unity by $\omega:=\exp(\frac{2\pi\mathrm{i}}{m})$. Furthermore, we set
\begin{align}
&\lambda_k:= \Re(\omega^k)=\cos\left(\frac{2\pi k}{m}\right),\qquad \mu_k:= \Im(\omega^k)=\sin\left(\frac{2\pi k}{m}\right),\nonumber\\
&v_k:=\frac{1}{m}\left(1,\omega^{-k},\omega^{-2k},\ldots,
\omega^{-(m-1)k}\right)^t\in\C^m,\quad 0\le k\le m-1. \label{eig_vek}
\end{align}
Note that $v_0= \frac{1}{m}\mathbf{1}:=\frac{1}{m}(1,1,\ldots,1)^t\in \R^m$.

The asymptotic distributional behavior of the sequence $(R_n)_{n\ge
  0}$ has  been identified in Janson \cite{Ja83,Ja04,Ja06}, see also
Pouyanne \cite{Pou05,Pou08} and, for the case $m=2$, Freedman \cite{Free65}. Janson developed a limit theory for the compositions of rather general urn schemes. For the cyclic urns  he showed that the normalized composition vector $R_n$ converges in distribution towards a multivariate normal distribution for $2\le m\le 6$, whereas for $m\ge 7$ there is no convergence by a conventionally standardized version of  the $R_n$ due to subtle periodicities. Further, for $m\ge 7$, there exists  a complex valued random variable $\Xi_1$ (depending on $m$) such that almost surely, as $n\to\infty$, we have
\begin{align}\label{strong}
\frac{R_n- \frac{n+1}{m}\mathbf{1}}{n^{\lambda_1}}
-2\Re\left(n^{i\mu_1}\Xi_1v_1\right) \to 0.
\end{align}

We focus mainly on the periodic case $m\ge 7$.  In
the present paper we study the fluctuations of  $n^{-\lambda_1}(R_n- \frac{n+1}{m}\mathbf{1})$ around the periodic sequence $(2\Re(n^{i\mu_1}\Xi_1v_1))_{n\ge 0}$.  We call the differences in (\ref{strong})  residuals.

To formulate our results we denote by $\stackrel{d}{\longrightarrow}$
convergence in distribution. Further, ${\cal N}(0,M)$ denotes the
centered normal distribution with covariance matrix $M$, where $M$ is
a symmetric positive semi-definite matrix. For $v\in\C^m$ we write
$v^*$ for the conjugate transpose of $v$ and for $z \in \C$, $\bar{z}$
denotes the complex conjugate of $z$.  Furthermore, $6\mid m$ and $6 \nmid m$ are short for $6$  divides (resp.~does not divide) $m$.

We distinguish the cases $6\mid m$ and $6 \nmid m$ as follows.
\begin{thm}\label{thm1}
Let $m \geq 2$ with $6\nmid m$ and set $r:=\lfloor (m-1)/6\rfloor$. Then, there exist complex valued random variables $\Xi_1,\ldots,\Xi_r$ such that, as $n \to \infty$, we have
\begin{equation*}
\frac{1}{\sqrt{n}}\left(R_n- \E[R_n] - \sum_{k=1}^r 2
\Re\left(n^{\omega^k}\Xi_k v_k\right)\right) \stackrel{d}{\longrightarrow}\mathcal{N}\left(0,\Sigma^{(m)}\right).
\end{equation*}
The covariance matrix $\Sigma^{(m)}$ has rank $m-1$ and is given by
\begin{equation*}
\Sigma^{(m)}= \sum_{k=1}^{m-1}\frac{1}{|2\lambda_k-1|} v_k v_k^*.
\end{equation*}
\end{thm}
When $6\mid m$ then the normalization requires an additional $\sqrt{\log n}$ factor and the rank of the covariance matrix is reduced to $2$:
\begin{thm}\label{thm2}
Let $m \geq 2$ with $6 \mid m$ and set $r:=\lfloor (m-1)/6\rfloor$. Then, there exist complex valued random variables $\Xi_1,\ldots,\Xi_r$  such that, as $n \to \infty$, we have
\begin{equation*}
\frac{1}{\sqrt{n\log n }}\left(R_n- \E[R_n] - \sum_{k=1}^r 2
\Re\left(n^{\omega^k}\Xi_k v_k\right) \right) \stackrel{d}{\longrightarrow}\mathcal{N}\left(0,\Sigma^{(m)}\right).
\end{equation*}
The covariance matrix $\Sigma^{(m)}$ has rank $2$ and is given by
\begin{equation*}
\Sigma^{(m)}= v_{m/6}v_{m/6}^*+ v_{5m/6}v_{5m/6}^*.
\end{equation*}
 \end{thm}

Note, that the sum $\sum_{k=1}^r$ in Theorem \ref{thm1} is empty for $2\le m\le 5$, also in Theorem \ref{thm2} for $m=6$. Hence, for $2\le m\le 6$ our theorems reduce to the central limit laws of Janson \cite{Ja83,Ja04,Ja06}. For $m\in\{7,8,9,10,11\}$ Theorem \ref{thm1} shows that there is a direct normalization of the residuals which implies a  multivariate central limit law (CLT).  The case $m=12$  also admits a multivariate CLT under a different scaling, see Theorem \ref{thm2}. For $m>12$ the residuals cannot directly be normalized to obtain convergence. However, Theorems \ref{thm1} and \ref{thm2} describe refined residuals which satisfy a multivariate CLT for all $m>12$. These can be considered as  asymptotic expansions of the random variables $R_n$.

The convergences  in Theorems \ref{thm1} and \ref{thm2} also hold for all moments. For an expansion of $\E[R_n]$ see (\ref{exp_mean}).

We conjecture Theorems \ref{thm1} and \ref{thm2} as being prototypical for a phenomenon  to occur frequently in related random combinatorial structures. E.g., we expect similar behavior for other urn models with analog almost sure random periodic behavior, see Janson \cite[Theorem 3.24]{Ja04}, further for the size of random $m$-ary search trees, cf.~\cite{chhw01}, or  for the number of leaves in random $d$-dimensional (point) quadtrees \cite{chfuhw07}. (For the latter two instances only the case of Theorem \ref{thm1} is expected to occur.)

The remainder of the present paper contains a proof of Theorems \ref{thm1} and \ref{thm2}. An outline of the proof is given in Section \ref{sec2}, where also the occurrence of the contributions $\Re(n^{\omega^k}\Xi_k v_k)$ in  Theorems \ref{thm1} and \ref{thm2} is  explained. Roughly, our proof combines a spectral decomposition of the residuals and estimates of their mixed moments with a recursive decomposition of the urn process and stochastic fixed-point arguments. In work in progress of the first mentioned author of the present paper also an alternative route via martingales is being explored. Within the details of the  proofs of the present paper we make mildly use of martingales. However, we could also work out the whole proof without drawing back to any martingale which may provide a useful general technique for related applications where no martingales are available.

The results of this paper were announced in the extended abstract \cite{mune16}.

\section{Explanation of the result and outline of the proof} \label{sec2}
In this section we set out our approach towards the proof of Theorems \ref{thm1} and \ref{thm2} and explain the occurrence of
the summands $\Re(n^{\omega^k}\Xi_k v_k
)$ and the normal fluctuation in the theorems.

We first recall known asymptotic behavior and a
spectral decomposition of
$R_n$ which are used subsequently. Then we state a more refined result
on certain projections of residuals in Proposition \ref{prop1} which
directly implies Theorems \ref{thm1} and \ref{thm2}. Finally, an outline
of the proof of Proposition \ref{prop1} is given. Technical steps and
estimates are then carried out in Section \ref{sec:3}. Throughout, we fix $m\ge 2$.

For the cyclic urn with $m$ colors we consider an initial configuration of one ball of type $0$ and write $R_n$ for the composition vector after $n$ steps. Its dynamics is summarized in the $m\times m$ replacement matrix
\begin{equation} \label{R}
{\cal A}:=
\begin{pmatrix}
0 & 1 & 0& \cdot & \cdot & 0 & 0 \\
0 & 0 & 1& \cdot & \cdot & 0 & 0 \\
0 & 0 & 0& \cdot & \cdot & \cdot & \cdot \\
\cdot & \cdot &  \cdot & \cdot & \cdot & \cdot & \cdot \\
\cdot & \cdot & \cdot & \cdot & \cdot & 0 & 1 \\
1 & 0 & 0& \cdot & \cdot & 0 & 0
\end{pmatrix},
\end{equation}
where ${\cal A}_{ij}$ indicates that after drawing a ball of type $i$ it is placed back together with ${\cal A}_{ij}$ balls of type $j$ for all $0\le i,j\le m-1$.
The canonical filtration is given by the $\sigma$-fields ${\cal F}_n=\sigma(R_0,\ldots,R_n)$ for $n\ge 0$.
The dynamics of the urn process imply the well-known almost sure relation
\begin{align}\label{bed_erw}
\mathbb{E}\left[R_{n+1} \,|\, \mathcal{F}_n \right]
= \sum_{k=0}^{m-1} \frac{R_{n,k}}{n+1} (R_{n} + {\cal A}^t e_k )
= \left(\mathrm{Id}_m + \frac{1}{n+1}{\cal A}^t \right)R_n,\qquad n\ge 0.
\end{align}
Here, $\mathrm{Id}_m$ denotes the $m\times m$ identity matrix and ${\cal A}^t$ the transpose of ${\cal A}$.

Note that $v_0$ has the direction of the drift vector $\mathbf{1}$ in Theorems \ref{thm1} and \ref{thm2} and, for $m\ge 7$, the vector $v_1$ determines  the direction of the a.s.~periodic fluctuations around the drift. By diagonalizing the matrices on the right hand side of (\ref{bed_erw}) one finds an exact asymptotic expression for the mean of $R_n$, cf.~\cite[Lemma 6.7]{knne14}.
With
\begin{equation*}
\xi_k := \frac{2}{\Gamma(1+\omega^k)}v_k,\quad 1\le k \le r,
\end{equation*}
equation (\ref{bed_erw}) implies the expansion, as $n\to\infty$,
\begin{equation}\label{exp_mean}
\mathbb{E}\left[R_n\right] = \frac{n+1}{m}\mathbf{1} +\sum_{k=1}^r \Re(n^{i\mu_k}\xi_k)n^{\lambda_k} + \left\{ \begin{array}{ll}
  o(\sqrt{n}),&\mbox{if } 6 \nmid m,\\
  \mathrm{O}(\sqrt{n}),&\mbox{if } 6 \mid m.
  \end{array}\right.
\end{equation}
It is also known  that the variances and covariances of the numbers of balls of each color are of the order $n^{2\lambda_1}$ when $m\ge 7$, with  appropriate periodic prefactors. This explains the normalization $n^{-\lambda_1}(R_n-\frac{n+1}{m}\mathbf{1})$ in (\ref{strong}).
The analysis of the asymptotic distribution as stated in (\ref{strong}) has been carried out by different techniques (partly only in a weak sense), by embedding into continuous time multitype branching processes, by (more direct) use of martingale arguments, and by stochastic fixed-point arguments, see \cite{Ja04,Pou05,knne14}.

\begin{figure}[tt]
\begin{center}
\includegraphics[width=6cm]{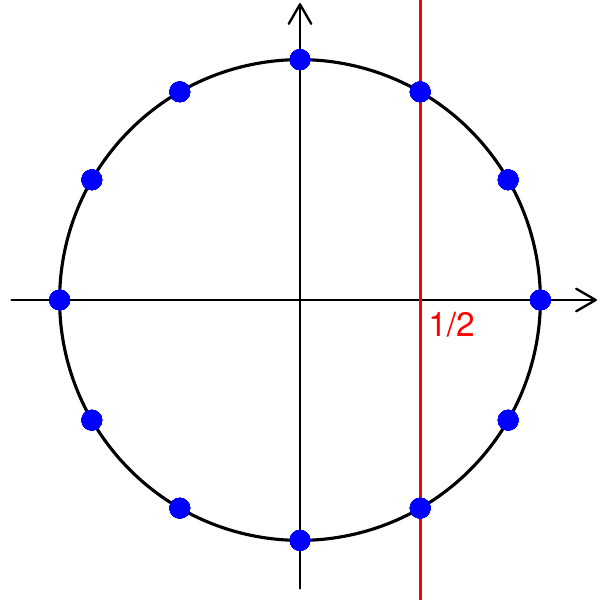}\qquad
\includegraphics[width=6cm]{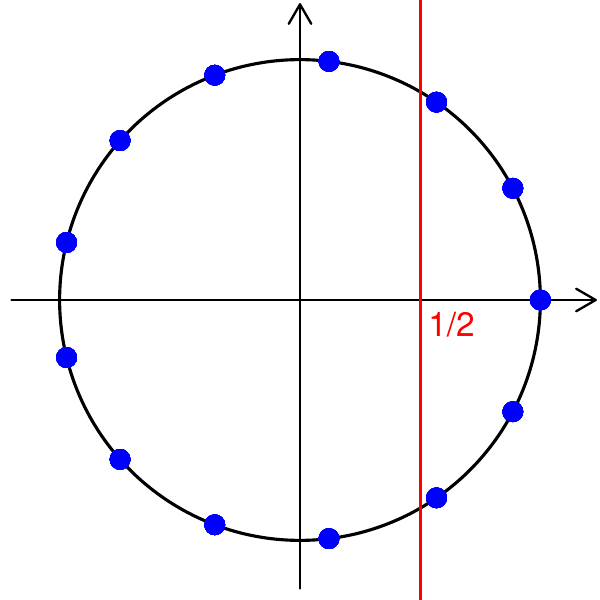}
\end{center}
\caption{Plots of the eigenvalues of $\mathcal{A}$ for $m=12$ (left
  picture) and $m=13$ (right picture). They correspond to
  contributions to $R_n$ as follows: The eigenvalue $\lambda_0=1$
  corresponds to the deterministic drift. All other eigenvalues with
  $\lambda_k>\frac{1}{2}$ correspond to almost sure periodic
  contributions with normal fluctuations around the periodic
  vector. The eigenvalues with $\lambda_k \le \frac{1}{2}$ correspond
  to contributions which only consist of a normal fluctuation. All
  normal fluctuations are of the same order if $6 \nmid m$. They
  compose an overall fluctuation of rank $m-1$, see Theorem
  \ref{thm1}. If $6\mid m$ then the eigenspaces with
  $\lambda_k=\frac{1}{2}$ contribute normal fluctuations of larger
  orders which dominate the contributions from all other
  eigenspaces. The overall fluctuations are then just the fluctuations
  from the two eigenspaces $m/6$ and $5m/6$ with $\lambda_{m/6}=\lambda_{5m/6}=\frac{1}{2}$ and of rank $2$, see Theorem \ref{thm2}.}
\end{figure}

For our further analysis we use a spectral decomposition of the process $(R_n)_{n\ge 0}$. This also leads to an explanation of the terms and fluctuations appearing in Theorems \ref{thm1} and \ref{thm2}, see the comments after the proof of Theorem \ref{thm2} in the present section.

We denote by $\pi_k$ the orthogonal projection onto the eigenspace in $\C^m$ spanned by $v_k$ for $0\le k\le m-1$. Hence, we have
\begin{align*}
R_n&=\pi_0(R_n)+\sum_{k=1}^{\lceil m/2\rceil -1} (\pi_{k}+\pi_{m-k})(R_n) + \ind_{\{ m \mbox{ even}\}}\pi_{m/2}(R_n) =\sum_{k=0}^{m-1} u_{k}(R_n) v_k,
\end{align*}
where $u_0 , \ldots, u_{m-1}$ denotes the basis dual to $v_0, \ldots, v_{m-1}$, as ${\cal A}$ is diagonizable. We have deterministically $\pi_0(R_n)=\frac{n+1}{m}\mathbf{1}$. For the other projections $\pi_k(R_n)$  one has  similar periodic behavior as for the composition vector $R_n$ in (\ref{strong}), as long as $\lambda_k>\frac{1}{2}$. Commonly, projections $\pi_k(R_n)$ are called `large', if $\lambda_k>\frac{1}{2}$, since their  magnitudes are larger than $\sqrt{n}$. Projections $\pi_k$ with $\lambda_k\le \frac{1}{2}$ are called `small'.

For large projections, i.e. for all $k\geq 1$ with
$\lambda_k>\frac{1}{2}$, we set
\begin{align} \label{grenz}
X_{n,k}:= \frac{1}{\sqrt{n}}
\left(
\begin{array}{c}
\Re\left(u_k(R_n - \E[R_n])-n^{\omega^k}\Xi_k\right) \\
 \Im\left(u_k(R_n - \E[R_n])-n^{\omega^k}\Xi_k\right)
\end{array} \right), \quad n \geq 1,
\end{align}
with an appropriate complex valued random variable $\Xi_k$, defined as a martingale limit in (\ref{mg_conv}), Section \ref{sec:31}. The behavior of the small
projections $\pi_k(R_n)$ has already been determined, see \cite{Ja04,mai14}. For those $k$ with $\lambda_k<\frac{1}{2}$ we have for $n \geq 1$
\begin{align}\label{rn1}
X_{n,k}:=\frac{1}{\sqrt{n}}
\left(
\begin{array}{c}
\Re(u_k(R_n - \E[R_n])) \\
 \Im(u_k(R_n - \E[R_n]))
\end{array} \right)
  \stackrel{\mathrm{d}}{\longrightarrow} {\cal N}
\left(0,\frac{\mathrm{Id}_2}{1-2\lambda_k}\right).
\end{align}
If $m$ is even, then for $n \geq 1$, $X_{n,m/2}:=n^{-1/2}u_{m/2}(R_n)
\stackrel{\mathrm{d}}{\longrightarrow} {\cal N}(0,1/3)$. For $m=2$,
the last mentioned result
has already been established by Freedman \cite[Theorem 5.1]{Free65}.

Finally, if $6 \mid m$, there are two eigenvalues with real parts equal to $\frac{1}{2}$. Compared to the other small components, the scaling of the associated projections requires an additional $\sqrt{\log n}$ factor for convergence: For $k\in\{m/6, 5m/6\}$ and $n \geq 1$,
\begin{align}\label{rn2}
X_{n,k}:=\frac{1}{\sqrt{n \log n}}
\left(
\begin{array}{c}
\Re(u_k(R_n - \E[R_n])) \\
 \Im(u_k(R_n - \E[R_n]))
\end{array} \right)
  \stackrel{\mathrm{d}}{\longrightarrow} {\cal N}\left(0,\frac{1}{2}\mathrm{Id}_2\right).
\end{align}
We prove the convergence of the variances and covariances of all $X_{n,k}$ in Section \ref{sec:31}. Set $X_{n,0}:= u_0(R_n - \E[R_n])= 0$ and $X_0 := (0, \ldots, 0)^t$.

To summarize, $X_{n,0}, \ldots, X_{n,m-1}$ describe the normalized fluctuations along the projections. For each pair of complex conjugate eigenvalues, there is one $X_{n,k}$ that captures the behaviour of the corresponding real and imaginary part. Small projections are known to be asymptotically normally distributed, see (\ref{rn1}). As a main contribution of the present paper we show that residuals of large projections as normalized in (\ref{grenz}) are also asymptotically normal. Moreover, fluctuations along different proections are asymptotically independent: 

\begin{prop}\label{prop1}
Assume that $6 \mid m$. For the vector $Z_n:=(X_{n,1}, \ldots,
X_{n,m/2}) \in \R^{m-1}$ defined for $n \geq 0$ in
(\ref{grenz})--(\ref{rn2}) we have, as $n\to\infty$, that \begin{align*}
Z_n \stackrel{d}{\longrightarrow} {\cal N}\left(0,M_m\right),
\end{align*}
with
\begin{align}\label{limitmatrix}
M_m:=\frac{1}{2} \mathrm{diag}\left(\frac{\mathrm{Id}_2}{|2\lambda_1-1|},\ldots, \frac{\mathrm{Id}_2}{|2\lambda_r-1|}, \mathrm{Id}_2, \frac{\mathrm{Id}_2}{|2\lambda_{r+2}-1|}, \ldots, \frac{\mathrm{Id}_2}{|2\lambda_{m/2-1}-1|}, \frac{2}{3}\right).
\end{align}
\end{prop}

In the case $6\nmid m$ Proposition \ref{prop1} holds as well. The only difference is that there is no $k$ with $\lambda_k=\frac{1}{2}$ and thus the matrix corresponding to $M_m$ for the case $6 \nmid m$ does not have the block $\frac{1}{2}\mathrm{Id}_2$. If $m$ is odd, also the last block $\frac{1}{3}$ is not present. 

Proposition \ref{prop1} (and its version for $6\nmid m$) directly imply Theorems \ref{thm1} and \ref{thm2}:

\proof[Proof of Theorem \ref{thm1}]
Note that $6 \nmid m$ implies that there is no $0\le k \le m-1$ with
$\lambda_k=\frac{1}{2}$. We obtain
\begin{align*}
\lefteqn{\frac{1}{\sqrt{n}}\left(R_n- \E[R_n] - \sum_{k=1}^r 2
\Re\left(n^{\omega^k}\Xi_k v_k\right)\right)}\\
&= \frac{1}{\sqrt{n}}\left(\sum_{k=1}^r \left\{2 \Re\left(\left[u_k(R_n-
  \E[R_n]) - n^{\omega^k}\Xi_k \right]v_k\right)\right\} \right.\\
&\left.\qquad\qquad\quad ~+ \sum_{k=r+1}^{\lceil m/2 \rceil-1}2\Re(u_k(R_n- \E[R_n])v_k) +\ind_{\{m \mbox{ even}\}} u_{m/2}(R_n- \E[R_n])v_{m/2}\right)\\
&= 2Z_{n,1}\Re(v_1) - 2 Z_{n,2}\Im(v_1) +
  \cdots+\ind_{\{m \mbox{ even}\}} Z_{n,m-1} v_{m/2}\\
&\stackrel{\mathrm{d}}{\longrightarrow} {\cal N}\left(0,\Sigma^{(m)}\right),
\end{align*}
by Proposition \ref{prop1} and the continuous mapping theorem, where
$\Sigma^{(m)}$ is as in the statement of Theorem \ref{thm1}. It is immediate that the image of $\Sigma^{(m)}$ in $\mathbb{R}^m$ is $\mathrm{span}\{\Re\left(v_1\right), \Im\left(v_1\right), \ldots, v_{m/2}\}$, if $2 \mid m$, and $\mathrm{span}\{\Re\left(v_1\right), \Im\left(v_1\right), \ldots, \Im\left(v_{(m-1)/2}\right)\}$ otherwise, hence the rank of $\Sigma^{(m)}$ is $m-1$.
\qed
\proof[Proof of Theorem \ref{thm2}]
Note that $6 \mid m$ implies that there is the pair  $\lambda_{m/6}=\lambda_{5m/6}=\frac{1}{2}$. Rearranging terms as in the proof of Theorem \ref{thm1} we obtain
\begin{align*}
&\frac{1}{\sqrt{n\log(n)}}\left(R_n- \E[R_n]- \sum_{k=1}^r 2
\Re\left(n^{\omega^k}\Xi_k v_k\right)\right) \\
&
= \frac{1}{\sqrt{\log(n)}}\sum_{k=1, k\not= m/6}^{m/2-1}2
  (Z_{n,2k-1}\Re(v_k) - Z_{n,2k}\Im(v_k))  \\
&+ 2
  (Z_{n,m/3-1}\Re(v_{m/6}) - Z_{n,m/3}\Im(v_{m/6})) +
 \frac{1}{\sqrt{\log(n)}} Z_{n,m-1} v_{m/2} \stackrel{\mathrm{d}}{\longrightarrow} {\cal N}\left(0,\Sigma^{(m)}\right),
\end{align*}
by Proposition \ref{prop1} and Slutsky's Lemma, where $\Sigma^{(m)}$
is as in Theorem \ref{thm2}. Again, it is immediate that the image of $\Sigma^{(m)}$ in $\mathbb{R}^m$ is $\mathrm{span}\{\Re\left(v_{m/6}\right),  \Im\left(v_{m/6}\right)\}$, hence its rank is $2$.
\qed

\vspace{0.5 cm}
The proofs of Theorems \ref{thm1} and \ref{thm2} via Proposition \ref{prop1} indicate the role of the terms $\Re(n^{\omega^k}\Xi_k v_k)$ in the overall Gaussian fluctuation, see also Figure 1: All eigenspaces with $\lambda_k>\frac{1}{2}$ (excluding the deterministic drift for $\lambda_k=1$) contribute two asymptotic components: First, there is the almost sure periodic component
\begin{align*}
\Re(n^{\omega^k}\Xi_k v_k) = n^{\lambda_k} \Re\left(\exp(i\mu_k\log n)\Xi_kv_k\right)
\end{align*}
of order $n^{\lambda_k}$ with a random periodic factor, periodic roughly in $\log n$. Second, there is a normal fluctuation (in distribution) of order $\sqrt{n}$. All eigenspaces with $\lambda_k<\frac{1}{2}$  add a contribution of order $\sqrt{n}$ to the normal fluctuation which is the visible order within these eigenspaces. For $6 \mid m$, there are eigenvalues with $\lambda_k=\frac{1}{2}$ and the normal fluctuation is of order $\sqrt{n \log n}$ in the corresponding two  eigenspaces. According to Proposition \ref{prop1} all these fluctuations within the eigenspaces are asymptotically independent, which explains the overall asymptotic normal fluctuation. Since this normal fluctuation is of order $\sqrt{n}$ and $\sqrt{n \log n}$, respectively, all the almost sure periodic contributions from the eigenspaces with $\lambda_k>\frac{1}{2}$ are visible as well.

\vspace{0.5 cm}
To prove Proposition \ref{prop1} we first derive moments and mixed moments in Section \ref{sec:31} needed for the normalization. In Section \ref{sec:33} a pointwise recursive equation
 for the complex random variables $\Xi_1,\ldots,\Xi_r$ is obtained together with a recurrence for the sequence $(R_n)_{n\ge 0}$ which extends to a recurrence for the residuals in (\ref{strong}) as well as to the residuals $Z_n$ of the projections  of the $R_n$, see equation (\ref{rec_zn}) in Section \ref{sec:33}. Equation (\ref{rec_zn}) is then the starting point to show the convergence in Proposition \ref{prop1}. For this, a stochastic fixed-point argument in the context of the contraction method within the Zolotarev metric $\zeta_3$, see \cite{neru04} for general reference, is used. Then, we draw back to an approach to bound the Zolotarev distance and some estimates from \cite{ne14}  where a related, but simpler, (univariate) problem was discussed.

\section{Proof of Proposition \ref{prop1}} \label{sec:3}
We start with estimates for the covariance matrix of the $Z_n$ appearing in Proposition \ref{prop1} in section \ref{sec:31}. In section \ref{sec:33} we derive the recurrence (\ref{rec_zn}) for the $Z_n$. The use of the Zolotarev metric $\zeta_3$ requires a slightly modified version of recurrence (\ref{rec_zn}). This is explained in section \ref{zolo_sec}, see in particular the quantities $N_n$ in (\ref{RekN}) which are the modified ´versions of the $Z_n$. Then in section  \ref{sec_technical} asymptotics for the coefficients appearing in the recurrence (\ref{rec_zn}) of $Z_n$ and $N_n$ respectively are derived. Based on these asymptotics finally in section \ref{sec:34} convergence of the $N_n$ is shown within the Zolotarev metric, which implies convergence in distribution of the $Z_n$ as stated in Proposition \ref{prop1}.

Recall that Proposition \ref{prop1} assumes that $6 \mid m$. As mentioned before, the analoguous result for $6 \nmid m$ is true and can be proved along the same lines by some minor modifications. 

\subsection{Convergence of the Covariance Matrix}\label{sec:31}
As indicated in Section \ref{sec2}, we study the centered process
$(R_n-\E[R_n])_{n \geq 0}$ via its spectral decomposition with respect
to the orthogonal basis $\{ v_k: 0 \leq k < m\}$  of the unitary
vector space $\C^m$, i.e.
\begin{align*}
 R_n - \mathbb{E}[ R_n] = \sum_{k=0}^{m-1} \pi_k\left(R_n - \E[R_n]\right)
=\sum_{k=0}^{m-1} u_k\left(R_n - \mathbb{E}[R_n ]\right)v_k,
\end{align*}
where $u_k\left(w\right):=  1 \cdot w_0 + \omega^k \cdot w_1 +
\cdots + \omega^{(m-1)k} \cdot w_{m-1}$ for $w \in \mathbb{C}^m$.
The evolution (\ref{bed_erw}) of the process implies that for $n \geq 1$, there is a complex normalization
\begin{align}\label{norm_rn_mkn}
M_{k,n} := \frac{\Gamma(n+1)}{\Gamma(n+1+\omega^k)} u_k\left(R_n - \mathbb{E}\left[R_n \right]\right) = \begin{cases}
\frac{\Gamma(n+1)}{\Gamma(n+1+\omega^k)} u_k\left(R_n\right) - \frac{1}{\Gamma(1+\omega^k)}, & k \not= m/2, \\
\frac{\Gamma(n+1)}{\Gamma(n+1+\omega^k)} u_k\left(R_n\right), &  k=m/2,
\end{cases}
\end{align}
that turns all the  eigenspace coefficients, $0\le k\le m-1$, into centered martingales. We set  $M_{k,0}:=0$. Depending
on $\lambda_k$, these
martingales are known to exhibit two different kinds of asymptotic
behavior, see \cite{Ja04,Ja06,Pou05}:
For all $k \in \{0, \ldots, m-1\}$ with $\lambda_k=\Re\left(\omega^k\right)>1/2$, there exists a complex valued random variable $\Xi_{k}$ such that, as $n \to \infty$, we have
\begin{align}\label{mg_conv}
M_{k,n} \to \Xi_k\; \mbox{ almost surely},
\end{align}
where the convergence also holds in $\mathrm{L}_p$ for every $p \geq
1$. Note that the $\Xi_k$ in (\ref{mg_conv}) are identical
with the $\Xi_k$ in (\ref{grenz}) and in Theorems \ref{thm1}
and \ref{thm2}.
The $M_{k,n}$ with
$\lambda_k=\Re\left(\omega^k\right)\le 1/2$ are known to converge in distribution, after proper normalization, to normal limit laws.

From Section \ref{sec:33} on, our analysis will also require to start the cyclic urn process with one ball of type $j\in\{0,\ldots,m-1\}$. The corresponding composition vector $R_n^{[j]}$ is obtained in distribution by the relation
\begin{align}\label{shift}
\left(R_n^{[j]}\right)_{n\ge 0} \stackrel{d}{=}\left(\left({\cal A}^t\right)^j R_n\right)_{n\ge 0},\quad 0\le j\le m-1,
\end{align}
with the replacement matrix ${\cal A}$ from (\ref{R}) and  where $\stackrel{d}{=}$ denotes equality in distribution. Similar to the identity (\ref{shift}), the corresponding martingales $M_{k,n}^{[j]}$ satisfy
\begin{equation*}
M_{k,n}^{[j+1]} \stackrel{d}{=} \omega^{k} M_{k,n}^{[j]},
\end{equation*}
with convention $M_{k,n}^{[m]}:= M_{k,n}^{[0]}$.

Our subsequent analysis requires asymptotics of moments and of correlations between the $u_k(R_n)$. Exploiting the dynamics of the urn in (\ref{bed_erw}), elementary calculations imply that:
\begin{lem}\label{erw}
For $k \in \{0, \ldots, m-1\}$, we have
\begin{equation*}
\mathbb{E}\left[u_k\left(R_n\right)\right] = \sum_{t=0}^{m-1}
\omega^{kt} \mathbb{E}\left[R_{n,t}\right] = \begin{cases}
\frac{\Gamma(n+1+\omega^k)}{\Gamma(n+1)\Gamma(1+\omega^k)}, \qquad k
\not= m/2,\\
0, \hspace{2.8 cm} k=m/2.
\end{cases}
\end{equation*}
For $k, \ell \in \{0, \ldots, m-1\}$,
\begin{eqnarray} \label{secmo}
\lefteqn{\mathbb{E}\left[u_{k}\left(R_n\right) u_{\ell}\left(R_n\right)\right]}\nonumber\\
&=&\prod_{s=1}^{n}\left(\frac{s+\omega^{k}+\omega^{\ell}}{s}\right)
~+ \sum_{s=1}^{n}\frac{\omega^{k+\ell}}{s}\prod_{t=1}^{s-1}\left( \frac{t+\omega^{k+\ell}}{t} \right)\prod_{t=s+1}^{n}\left(\frac{t+\omega^{k}+\omega^{\ell}}{t}\right).
\end{eqnarray}
\end{lem}

\begin{proof}
The first two identities immediately follow from (\ref{bed_erw}). For
(\ref{secmo}), let $k, \ell \in \{0, \ldots, m-1\}$ and $n \geq 1$ and
note that, almost surely,
\begin{align*}
\E\left[ u_k(R_n)u_{\ell}(R_n) | \mathcal{F}_{n-1}\right] &= \left(1+\frac{\omega^k+\omega^{\ell}}{n}\right) u_k(R_{n-1})u_{\ell}(R_{n-1}) + \frac{\omega^{k+\ell}}{n}u_{k+\ell}(R_{n-1}).
\end{align*}
Here, we use the abbreviation $u_{k+\ell}(R_{n-1}):= u_{(k+\ell)\mod m}(R_{n-1})$.
\end{proof}
\noindent \textbf{Remark 1.} From (\ref{secmo}) we see that all
$\E[|u_k(R_n)|^2]$ with $\lambda_k <1/2$ are of linear order,
all $\E[|u_{k}(R_n)|^2]$ with $\lambda_k =1/2$ are of order $n \log n$ and all
$\E[|u_k(R_n)|^2]$
with $\lambda_k >1/2$ have order $n^{2\lambda_k}$. To make this more visible from (\ref{secmo}), we make some case distinctions.

We first consider the real cases $k=\ell=0$ and $k=\ell =m/2$ for $2 \mid m$:
\begin{align*}
\E\left[|u_{0}(R_n)|^2\right] &= (n+1)^2
\end{align*}
and, if $2 \mid m$,
\begin{align*}
\E\left[|u_{m/2}(R_n)|^2\right] &= \frac{n+1}{3}.
\end{align*}
Now, $\omega^k + \omega^{\ell}=-1$ only if $3 \mid m$ and $\{k, \ell\} = \{m/3, 2m/3\}$. In this case,
\begin{align*}
\E\left[|u_{m/3}(R_n)|^2\right] &= \frac{1}{n} \sum_{t=1}^{n} t = \frac{n+1}{2}.
\end{align*}
On the other hand, $\omega^{k+\ell} = \omega^k + \omega^{\ell}$ only if $6 \mid m$ and $\{k, \ell\} = \{m/6, 5m/6\}$. In this case, $\omega^k + \omega^{\ell}= 1$ and
\begin{align*}
\E\left[|u_{m/6}(R_n)|^2\right] &= (n+1) \sum_{t=1}^{n+1} \frac{1}{t} \sim n\log n.
\end{align*}
Thirdly, $\omega^k+\omega^{\ell}=0$ if and only if $2 \mid m$ and $\ell = k + m/2 \mod m$, so in this case
\begin{align*}
\E\left[u_{k}(R_n)u_{\ell}(R_n)\right] &=
\begin{cases}
0, \hspace{3.1 cm} \text{if} \quad \{k,l\} = \{0, m/2\}, \\
\frac{\Gamma(n+1+\omega^{k+\ell})}{\Gamma(n+1)\Gamma(1+\omega^{k+\ell})},
\qquad \text{else}.
\end{cases}
\end{align*}
Finally, $\omega^{k+\ell} = -1$ if $\lambda_k = - \lambda_{\ell}$ and
$\mu_k=\mu_{\ell}$ and then,
\begin{align*}
\E\left[ u_k(R_n)u_{\ell}(R_n) \right] &= \frac{\omega^k + \omega^{\ell}}{1 + \omega^k + \omega^{\ell}} \prod_{s=1}^n \left(1+\frac{\omega^k+\omega^{\ell}}{s}\right) \sim \frac{\omega^k + \omega^{\ell}}{\Gamma(2 + \omega^k + \omega^{\ell})} n^{\omega^k + \omega^{\ell}}.
\end{align*}
In all other cases,
\begin{align*}
\E\left[ u_k(R_n)u_{\ell}(R_n) \right]
&=\frac{1}{\omega^{k+\ell}-\omega^k-\omega^{\ell}}\left(\frac{\Gamma(n+1+\omega^{k+\ell})}{\Gamma(n+1)\Gamma(\omega^{k+\ell})} - \frac{\Gamma(n+1+\omega^k+\omega^{\ell})}{\Gamma(n+1)\Gamma(\omega^k+\omega^{\ell})} \right).
\end{align*}

\vspace{0.3 cm}
\noindent \textbf{Remark 2.} From (\ref{secmo}) we obtain the mixed moments of the corresponding real and imaginary parts via the identities
\begin{align*}
\E\left[ \Re(u_k(R_n))\Re(u_{\ell}(R_n) )\right] &= \frac{1}{2} \Re\left( \E\left[ u_k(R_n)u_{\ell}(R_n) \right] + \E\left[ u_k(R_n)u_{m-\ell}(R_n) \right]\right),\\
\E\left[ \Im(u_k(R_n))\Im(u_{\ell}(R_n) )\right] &= \frac{1}{2} \Re\left( \E\left[ u_k(R_n)u_{m-\ell}(R_n) \right] - \E\left[ u_k(R_n)u_{\ell}(R_n) \right]\right),\\
\E\left[ \Re(u_k(R_n))\Im(u_{\ell}(R_n) )\right] &= \frac{1}{2} \Im\left( \E\left[ u_k(R_n)u_{\ell}(R_n) \right] + \E\left[ u_{m-k}(R_n)u_{\ell}(R_n) \right]\right).
\end{align*}

\vspace{0.3 cm}

From Lemma \ref{erw}  we obtain the order of magnitude of the
$\mathrm{L}_2$-distance of the residuals of the martingales
$(M_{k,n})_{n\ge 0}$ with $\lambda_k>\frac{1}{2}$. This is needed for the proper normalization of these residuals.
\begin{lem} \label{mgconv}
For $k \geq 1$ such that $1/2 < \lambda_k < 1$ and $\Xi_k$ as in (\ref{mg_conv}), as $n \to \infty$,
\begin{equation*}
 \E\left[\left|M_{k,n} - \Xi_k\right|^2\right] \sim \frac{1}{2 \lambda_k -1}n^{1-2\lambda_k}
\end{equation*}
and
\begin{equation*}
 \E\left[\left(M_{k,n} - \Xi_k\right)^2\right] \sim \frac{1}{(1-2\omega^{-k})\Gamma(2\omega^k)}n^{-1}.
\end{equation*}
In particular,
\begin{align*}
\E\left[\Re\left(M_{k,n} - \Xi_k\right)^2\right] &\sim  \frac{1}{2}\frac{1}{2 \lambda_k -1}n^{1-2\lambda_k}, \\ \E\left[\Im\left(M_{k,n} - \Xi_k\right)^2\right] &\sim  \frac{1}{2} \frac{1}{2 \lambda_k -1}n^{1-2\lambda_k},\\
\E\left[\Re\left(M_{k,n} - \Xi_k\right) \Im\left(M_{k,n} - \Xi_k\right)\right] &\sim  \frac{1}{2} \Im\left( \frac{1}{(1-2\omega^{-k})\Gamma(2\omega^k)}\right) n^{-1}.
\end{align*}
\end{lem}

\begin{proof}
We show the claim for $\E\left[\left|M_{k,n} - \Xi_k\right|^2\right] $
in an exemplary way. Here, we decompose
\begin{align*}
\E\left[\left|M_{k,n} - \Xi_k\right|^2\right] &= \sum_{z=n}^{\infty} \E\left[\left|M_{k,z} - M_{k,z+1}\right|^2\right] \\
&= \sum_{z=n}^{\infty} \left|\frac{\Gamma(z+2)}{\Gamma(z+2+\omega^k)}\right|^2\E\left[\left| u_k(R_{z+1}-R_z)-\frac{\omega^k}{z+1}u_k(R_z)\right|^2\right]\\
&= \sum_{z=n}^{\infty} \left|\frac{\Gamma(z+2)}{\Gamma(z+2+\omega^k)}\right|^2 \left(\E\left[\left| u_k(R_{z+1}-R_z)\right|^2\right] - \frac{1}{(z+1)^2}\E\left[\left| u_k(R_{z})\right|^2\right] \right)\\
&=  \sum_{z=n}^{\infty} \left|\frac{\Gamma(z+2)}{\Gamma(z+2+\omega^k)}\right|^2 \left(1 + \frac{1}{1-2\lambda_k}\frac{1}{(z+1)^2} \left(\frac{\Gamma(z+1+2\lambda_k)}{\Gamma(z+1)\Gamma(2\lambda_k)} - z-1\right)\right)\\
&\sim \sum_{z=n}^{\infty} z^{-2\lambda_k} \sim \frac{1}{2\lambda_k-1} n^{1-2\lambda_k}
\end{align*}
as $n \to \infty$.
\end{proof}

The preceding calculations imply that the covariance matrix of $Z_n$, see Proposition \ref{prop1}, converges as $n \to \infty$. Its limit is given by $M_m$ defined in (\ref{limitmatrix}).

\subsection{Embedding and Recursions}\label{sec:33}
In this section we briefly explain how to derive an almost sure recurrence for the sequence $(R_n)_{n\ge 0}$ which then
extends to the projections. These recursive representations transfer to the martingale limits $\Xi_k$ and thus also to the components of $Z_n$.

We embed the cyclic urn process into a random binary search tree
generated by a sequence $(U_n)_{n \ge 1}$ of i.i.d.~random variables,
where $U:=U_1$ is uniformly distributed on $[0,1]$. The random
binary search tree starts with one external node at time $0$, the
so-called root. At time
$n=1$, the first key $U$ is inserted in this external node, turning
it into an internal node. The occupied node then grows two external nodes attached along a left and right branch. We
successively insert the following keys, where each key traverses the
internal nodes starting at the root, which is occupied by $U$. Whenever the key
traversing is less than the occupying key at a node it moves on to the
left child of that node, otherwise to its right child. The first
external node visited is occupied by the key, turning it into an
internal node with two new external nodes attached. It is easy to see that in each step one of the external nodes is chosen
uniformly at random (and independently of the previous choices) and
replaced by one internal node with two new external nodes attached. See, e.g., Mahmoud \cite{Ma92}, for a detailed description of random binary search trees.

The cyclic urn is embedded into the evolution of the random binary search tree by
labeling its external nodes by the types of the balls. The initial
external node is labeled by type $0$. Whenever an  external node of
type $j\in\{0,\ldots,m-1\}$ is replaced by an internal node then its new left external node is labeled $j$ (corresponding to returning the chosen ball of type $j$ to the urn) and its new right external node is  labeled $(j+1)\mod m$ (corresponding to the addition of a new ball of type $(j+1)\mod m$ to the urn). A related embedding was exploited in \cite[Section 6.3]{knne14}, see also \cite{chmapo15}. Note that the binary search tree starting with one external node labeled $0$ decomposes into its left and right subtree starting with external nodes of types $0$ and $1$, respectively. The size (number of internal nodes) $I_n$ of the left subtree is uniformly distributed on $\{0,\ldots,n-1\}$ and, conditional on $U=u$, $u\in (0,1)$, it is binomial $B_{n-1,u}$ distributed. This implies, with $J_n:=n-1-I_n$, the recurrence
\begin{align}\label{basic_rec}
R_n^{[0]} = R_{I_n}^{[0],(0)} + R_{J_n}^{[1],(1)}= R_{I_n}^{[0],(0)} +
  {\cal A}^t R_{J_n}^{[0],(1)}, \quad n \geq 1,
\end{align}
where the sequences $(R_{n}^{[0],(0)})_{n\ge 0}$ and $(
R_{n}^{[1],(1)})_{n\ge 0}$ denote the composition vectors of the
cyclic urns given by the evolutions of the left and right subtrees of
the root of the binary search tree (upper indices $[0]$ and $[1]$
denoting the initial type, upper indices $(0)$ and $(1)$  denoting
left and right subtree). They are independent of
$I_n$. We have set $(R_n^{[0],(1)})_{n \ge 0} := ({\cal A}
R_n^{[1],(1)})_{n \ge 0}$, and note that due to identity (\ref{shift}),
  $(R_n^{[0],(1)})_{n \ge 0}$ is a cyclic urn process started with one
  ball of type $0$ at time $0$. Now, applying the  transformation and
  scaling (\ref{norm_rn_mkn}) which turn $R_n$ into  $M_{k,n}$ to the
  left and right hand side of (\ref{basic_rec}), letting $n\to\infty$
  and using the convergence in (\ref{mg_conv}) yields the following almost sure recursive equation for the  $\Xi_k$:
\begin{prop}
For all $k\ge 1$ with $\lambda_k>\frac{1}{2}$ there exist random variables $\Xi_k^{(0)}$, $\Xi_k^{(1)}$ such that
\begin{align}\label{rec_xik}
\Xi_k = U^{\omega^k} \Xi_k^{(0)} + \omega^k (1-U)^{\omega^k} \Xi_k^{(1)} + g_k(U),
\end{align}
$U, \Xi_k^{(0)}$, $\Xi_k^{(1)}$ are independent, $U$ is uniformly distributed on $[0,1]$ and $\Xi_k^{(0)}$ and $\Xi_k^{(1)}$ have the same distribution as $\Xi_k$.
Here,
\begin{equation*}
 g_k(u):=\frac{1}{\Gamma(1+\omega^k)}\left(u^{\omega^k} + \omega^k(1-u)^{\omega^k} -1 \right).
\end{equation*}

\end{prop}
Here and subsequently, we make no use of the fact that the martingale limits $\Xi_k$ can also be written explicitly as deterministic functions of the limit of the random binary search tree when interpreting the evolution of the random binary search tree as a transient Markov chain and its limit as a random variable in the Markov chain's Doob-Martin boundary, see \cite{evgrwa12,gr14}. Following this path the $\Xi_k$ become a deterministic function of $(U_n)_{n\ge 1}$  and from this representation the self-similarity relation (\ref{rec_xik}) can be read off as well. See  \cite{BiFi12} for a related explicit construction.

Returning to $Z_n$, we see that
\begin{align}\label{rec_zn}
Z_n = \sigma_{I_n}^{-1} \sigma_n Z_{I_n}^{(0)} +
  \sigma_{J_n}^{-1}\sigma_{n}\mathcal{D} Z_{J_n}^{(1)} + \sigma_n
  F_n, \quad n \geq 1,
\end{align}
where $\sigma_0:=\sigma_1:=\mathrm{Id}_{m-1}$ and $\sigma_k:=\frac{1}{\sqrt{k}}\mathrm{diag}\left(1, \ldots, 1,
\frac{1}{\sqrt{\log k}}, \frac{1}{\sqrt{\log k}}, 1, \ldots,
1\right)$ for $k\ge 2$, where the additional factor of
$\sqrt{\log k}$ is needed for the eigenspace $m/6$ (recall that  $\lambda_{m/6}=\frac{1}{2}$),
the ${(m-1) \times (m-1)}$ matrix $\mathcal{D}$ is composed of
rotation matrices
\begin{eqnarray*}
\mathcal{D}=\left(
  \begin{array}{ccccccc}
        & \cos\left( \frac{2\pi}{m}\right)  & -\sin\left( \frac{2\pi}{m}\right) & &&&\\
         &  \sin\left( \frac{2\pi}{m}\right)& \cos\left( \frac{2\pi}{m}\right) & &&&\\
     & & &  \ddots   &&&\\
&&&&\cos\left( \frac{2\pi (m/2-1)}{m}\right)&-\sin\left( \frac{2\pi (m/2-1)}{m}\right)&\\
&&&&\sin\left( \frac{2\pi (m/2-1)}{m}\right)&\cos\left( \frac{2\pi (m/2-1)}{m}\right)&\\
 & & &  & &&-1\\
  \end{array}\right)
\end{eqnarray*}
and the ``error term'' $F_n$ is made up of three components: Setting
\begin{align}
G_{k,n}(\ell) &:= \frac{\Gamma(\ell + 1 + \omega^k)}{\Gamma(\ell + 1) \Gamma(1+\omega^k)} + \omega^k \frac{\Gamma((n-1-\ell)+1+\omega^k)}{\Gamma((n-1-\ell)+1) \Gamma(1+\omega^k)} - \frac{\Gamma(n+1+\omega^k)}{\Gamma(n+1)\Gamma(1+\omega^k)}
\end{align}
for $\ell \in \{0, \ldots, n-1\}$, we have $F_n = F_n^{(1)} +
F_n^{(2)}$, where
\begin{eqnarray*}
F_n^{(1)}& := \left(
  \begin{array}{c}
\Re\left(G_{1,n}(I_n)\right) \\
\Im\left(G_{1,n}(I_n) \right) \\
      \vdots \\
\Re\left(G_{r,n}(I_n) \right) \\
\Im\left(G_{r,n}(I_n)\right) \\
\Re\left(G_{r+1,n}(I_n)\right) \\
\Im\left(G_{r+1,n}(I_n) \right) \\
\vdots \\
0
  \end{array}\right) - \left(
\begin{array}{c}
\Re\left(n^{\omega} g_1(U)\right) \\
\Im\left(n^{\omega} g_1(U)\right) \\
      \vdots \\
\Re\left(n^{\omega^r} g_{r} (U)\right) \\
\Im\left(n^{\omega^r} g_{r} (U)\right) \\
0\vspace{-2mm}\\
\vdots \\
0
\end{array} \right),
\end{eqnarray*}
and $F_n^{(2)}$ is given by the sum
\begin{eqnarray*}
\left(
  \begin{array}{c}
\Re\left(\left( I_n^{\omega} -
    (nU)^{\omega}\right)
    \Xi_1^{(0)} + \left( J_n^{\omega}- (n(1-U))^{\omega}\right) \omega \Xi_1^{(1)}\right)\\
\Im\left(\left( I_n^{\omega} -
    (nU)^{\omega}\right)
    \Xi_1^{(0)} +\left( J_n^{\omega} - (n(1-U))^{\omega}\right) \omega\Xi_1^{(1)}\right) \\
\vdots \\
\Re\left(\left( I_n^{\omega^r} -
    (nU)^{\omega^{r}}\right)
    \Xi_{r}^{(0)}+ \left( J_n^{\omega^r}- (n(1-U))^{\omega^{r}}\right) \omega^{r}\Xi_{r}^{(1)}\right)  \\
\Im\left(\left( I_n^{\omega^r}-
    (nU)^{\omega^{r}}\right)
    \Xi_{r}^{(0)}+ \left( J_n^{\omega^r}- (n(1-U))^{\omega^{r}}\right) \omega^{r} \Xi_{r}^{(1)}\right)\\
0\vspace{-2mm}\\
\vdots \\
0
  \end{array}\right) .
\end{eqnarray*}
Note that $\mathcal{D} M_m \mathcal{D}^t = M_m$.

\subsection{The Zolotarev metric}\label{zolo_sec}

In the last subsection, we prepared a proof of Proposition
\ref{prop1} that is based on the contraction method. To be more
precise, weak convergence in Proposition \ref{prop1} is shown by (the stronger) convergence within the Zolotarev metric.  The Zolotarev metric has been
studied systematically in the context of distributional recurrences in
\cite{neru04}. We only give the definitions of the relevant quantities
and properties here.

For $x \in \R^{d}$, we denote by $\| x\|$ the standard Euclidean norm of
$x$, and for $B \in \R^{d \times d}$, $\|B\|_{\text{op}}$ denotes the
corresponding operator norm. For random variables $X$ and $p \geq 1$,
we denote by $\|X\|_p$ the L$_p$-norm of $X$.

For two $\R^{d}$ valued random variables $X$ and $Y$ we set
\begin{align*}
\zeta_3(X,Y) := \sup_{f \in \mathcal{F}_3}|\E[f(X)-f(Y)]|,
\end{align*}
where
\begin{align*}
\mathcal{F}_3 := \left\{ f \in C^2(\R^{d}, \R): \|D^2f(x) - D^2f(y)\|_{\text{op}} \leq \|x-y\| , \hspace{0.2 cm} x, y \in \R^{d} \right\}.
\end{align*}
We call a pair $(X,Y)$ $\zeta_3$-compatible if the expectation and the covariance matrix of $X$ and $Y$ coincide and if both $\|X\|_{3}, \|Y\|_{3} < \infty$. This implies that $\zeta_3(X,Y) < \infty$. A basic property is that $\zeta_3$ is $(3,+)$-ideal, i.e.,
\begin{eqnarray*}
\zeta_3(X+Z,Y+Z)\le\zeta_3(X,Y), \quad  \zeta_3(cX,cY) = c^3 \zeta_3(X,Y)
\end{eqnarray*}
for random vectors $X,Y,Z$, where $Z$ is independent of $X,Y$ and
$c>0$. For a linear transformation $A$ of $\R^{d}$, we have
\begin{eqnarray}\label{op_raus}
\zeta_3(AX,AY) \le \|A\|_\mathrm{op}^3 \zeta_3(X,Y).
\end{eqnarray}
The following lemma will be used in the proof of Proposition \ref{prop1} and can be proved similarly to Lemma $2.1$ in \cite{ne14}.
\begin{lem}\label{zolosum}
Let $V_1, V_2, W_1, W_2$ be random variables in $\R^{d}$ such that
$(V_1, V_2)$ and $(V_1 + W_1, V_2+W_2)$ are $\zeta_3-$compatible. Then
we have
\begin{align*}
\zeta_3(V_1+W_1, V_2+W_2) \leq \zeta_3(V_1,V_2) + \sum_{i=1}^2 \left(\|V_i\|_{3}^2\|W_i\|_{3} + \frac{\|V_i\|_{3}\|W_i\|_{3}^2}{2} + \frac{\|W_i\|_{3}^3}{2}\right).
\end{align*}
\end{lem}

In order to work with the Zolotarev metric later, it is necessary to
adjust the covariance matrix of $Z_n$. I.e., we need to work with a
sequence of random vectors that is sufficiently close to $(Z_n)_{n
  \geq 0}$ and
has fixed covariance matrix $M_m$ to guarantee the finiteness of the
corresponding Zolotarev distances $\zeta_3$.

As noted in section \ref{sec:31}, the covariance matrices
$(\Cov(Z_n))_{n \ge 0}$ converge componentwise
to $M_m$, and $M_m$ is invertible.
Thus, there exists $n_0 \in \N$ such that for all $n \geq
n_0$, $\Cov(Z_n)$ is invertible.
Defining
\begin{align}\label{def_Si}
\Sigma_n := \ind_{\{n < n_0\}} \mathrm{Id}_m + \ind_{\{n \geq n_0\}} M_m^{1/2} \Cov(Z_n)^{-1/2},
\end{align}
$\Sigma_n$ is invertible for all $n \geq 0$ and we see that $\Sigma_n Z_n$ has covariance matrix $M_m$ for all $n \geq n_0$. We now set
\begin{align}\label{RekN}
N_n := \Sigma_n Z_n = A_n^{(0)} N_{I_n}^{(0)} + A_n^{(1)} N_{J_n}^{(1)} + b_n,
\end{align}
where the right hand side is a recursive decomposition of $N_n$ with coefficients
\begin{align*}
A_n^{(0)}:= \Sigma_n \sigma_n\sigma_{I_n}^{-1}\Sigma_{I_n}^{-1},
  \hspace{0.2 cm} A_n^{(1)}:= \Sigma_n \sigma_{n} \sigma_{J_n}^{-1} \mathcal{D}\Sigma_{J_n}^{-1}, \hspace{0.2 cm} b_n:= \Sigma_n \sigma_n\left( F_n^{(1)} + F_n^{(2)} \right).
\end{align*}

\subsection{Preparatory Lemmata}\label{sec_technical}

In this section we collect some technical lemmata needed in the proof of Proposition \ref{prop1} in the next section. We first look at the asymptotics of the coefficients arising in recursion (\ref{RekN}).

\begin{lem} \label{Acoeff}
For all $1 \leq p < \infty$, as $n \to \infty$,
\begin{align*}
\left\| A_n^{(0)} - \sqrt{U} \cdot \mathrm{Id}_{m-1}\right\|_{p} \to 0
  \hspace{0.6 cm} \text{and} \hspace{0.6 cm} \left\| A_n^{(1)} -
  \sqrt{1-U} \cdot \mathcal{D} \right\|_{p} \to 0.
\end{align*}
\end{lem}

\begin{proof}
We first check almost sure convergence. Both $\sqrt{I_n/n}, \sqrt{(I_n\log{I_n})/(n\log{n})} \to \sqrt{U}$ and
$\sqrt{J_n/n}, \sqrt{(J_n\log{J_n})/(n\log{n})} \to \sqrt{1-U}$
a.s. as $n \to \infty$. Also, because $I_n \to
\infty$ a.s. as $n \to \infty$, both $\Sigma_n, \Sigma_{I_n}^{-1} \to \mathrm{Id}_{m-1}
$. The claim now follows for all $1 \leq p < \infty$ by an
application of the dominated convergence theorem.
\end{proof}

\begin{lem} \label{GamU}
Let $k \in \{1, \ldots, r\}$. As $n \to \infty$,
\begin{align*}
\left\|\left( \frac{I_n}{n}\right)^{\omega^k} - U^{\omega^k} \right\|_{3} = \bo\left( n^{-\lambda_k/2}\right).
\end{align*}
\end{lem}

\begin{proof}
The triangle inequality implies
\begin{align}\label{tri_est}
\left\|\left( \frac{I_n}{n}\right)^{\omega^k} - U^{\omega^k}\right\|_{3} &\leq \left\|\left( \frac{I_n}{n}\right)^{\lambda_k} - U^{\lambda_k}\right\|_{3} + \mu_k \left\|\left( \frac{I_n}{n}\right)^{\lambda_k} \log\left(\frac{I_n}{nU} \right)\right\|_{3}.
\end{align}
We start by considering the first summand in the latter display. Denoting by $B_{n-1,U}$ a mixed binomial distribution with parameters $n-1$ and $U$, we see that
\begin{align*}
\left\|\left( \frac{I_n}{n}\right)^{\lambda_k} - U^{\lambda_k}\right\|_{3} & \leq \left\|\left( \frac{I_n}{n}\right) - U\right\|^{\lambda_k}_{3} = \E\left[ \E\left[\left|\frac{B_{n-1,U}}{n} -U \right|^3 | U \right]\right]^{\frac{\lambda_k}{3}}
\end{align*}
 since  $I_n$, conditional on $U=u$, has the $B_{n-1,u}$ distribution. Employing the Marcinkiewickz--Zygmund inequality, there exists a constant $C$ independent of $u \in [0,1]$ such that
\begin{align*}
\E\left[ \left| B_{n-1,u} - (n-1)u\right|^3\right] & \leq C n^{\frac{3}{2}}.
\end{align*}
This implies $\left\|\left( \frac{I_n}{n}\right)^{\lambda_k} - U^{\lambda_k}\right\|_{3} = \bo\left( n^{-\lambda_k/2}\right)$. For the analysis of the second summand in (\ref{tri_est}), we also condition on $U$ and write
\begin{align*}
\left\|\left( \frac{I_n}{n}\right)^{\lambda_k} \log\left(\frac{I_n}{nU} \right)\right\|^3_{3} &= \int_0^1\E\left[ \left| \left( \frac{B_{n-1,u}}{n}\right)^{\lambda_k} \log\left(\frac{B_{n-1,u}}{nu} \right)\right|^3\right]du
\end{align*}
We divide the integral into two parts. For this purpose, define E$_u:=\{B_{n-1,u}\geq \frac{un}{e}\}$. Chernoff's inequality implies that for $0 \leq t < u(n-1)$
\begin{align*}
\Prob\left(B_{n-1,u} - u(n-1) < -t \right) & \leq \exp(-t^2/(2u(n-1))),
\end{align*}
so the complement E$_u^c$ of E$_u$ satisfies $\Prob($E$_u^c) \leq
\exp(-C_0 u n)$ for some constant $C_0 > 0$. We further denote by $h_{\lambda_k}: [0,\infty) \to \R$ the function $h_{\lambda_k}(x):=x^{\lambda_k}\log(x)$ (convention: $0 \cdot \log 0 := 0$). Then $\sup_{x \in [0,1]}|h_{\lambda_k}(x)| = \frac{1}{\lambda_k e} < \frac{2}{e} < 1$. We can now bound the expectation on E$_u^c$ in the following way:
\begin{align*}
\E\left[ \left| \left( \frac{B_{n-1,u}}{n}\right)^{\lambda_k} \log\left(\frac{B_{n-1,u}}{nu} \right)\right|^3 \ind_{\text{E}_u^c}\right] &= \int_{\text{E}^c_u}u^{3 \lambda_k}\left| h_{\lambda_k}\left(\frac{B_{n-1,u}}{un}\right)\right|^3 d\Prob \leq u^{3 \lambda_k}  \exp \left(-C_0 un\right).
\end{align*}
On E$_u$, we apply the mean value theorem to
$h_1\left((1+y)^{\lambda_k}\right) - h_1\left(1^{\lambda_k}\right))$ with $y=\frac{B_{n-1,u}-nu}{nu}$. Note
that $(\min\{1, 1+y \}, \max\{1, 1+y\}) \subset
[\frac{1}{e},\frac{1}{u}]$ on E$_u$ and that $|h_1'|$ is nonnegative
and increasing on this interval. Thus,
\begin{align*}
\E &\left[ \left| \left( \frac{B_{n-1,u}}{n}\right)^{\lambda_k}
  \log\left(\frac{B_{n-1,u}}{nu} \right)\right|^3
  \ind_{\text{E}_u}\right] \\
&= \int_{\text{E}_u}\left(\frac{1}{\lambda_k}\right)^3 u^{3 \lambda_k} \left|
                             h_1\left(\left(1+ \frac{B_{n-1,u}-nu}{nu} \right)^{\lambda_k}\right)
                             - h_1\left(1^{\lambda_k}\right)\right|^3 d \Prob \\
&\leq \int_{\text{E}_u}\left(\frac{1}{\lambda_k}\right)^3 u^{3
  \lambda_k} \left(\sup_{v \in [\frac{1}{e},\frac{1}{u}]}|h_1'(v)| \right)^3\left|
  \left(1+ \frac{B_{n-1,u}-nu}{nu} \right)^{\lambda_k} -
  1^{\lambda_k}\right|^3 d \Prob\\
&\leq \int_{\text{E}_u}\left(\frac{1}{\lambda_k}\right)^3 u^{3
  \lambda_k} \left(h_1'\left(\frac{1}{u}\right)\right)^3 \left|
  \frac{B_{n-1,u}-nu}{nu} \right|^{3\lambda_k} d \Prob\\
& \leq \left(\frac{1}{\lambda_k}\right)^3 n^{-3\lambda_k} (1-\log(u))^3\E\left[|-u+B_{n-1,u}-(n-1)u|^3 \right]^{\lambda_k}\\
& \leq C_k \hspace{0.2 cm} \frac{(1-\log(u))^3}{n^{3\lambda_k/2}}
\end{align*}
for some constant $C_k >0$. Combining these estimates, we obtain
\begin{align*}
\left\|\left( \frac{I_n}{n}\right)^{\lambda_k}
  \log\left(\frac{I_n}{nU} \right)\right\|^3_{3} &\leq \int_0^1\left(
                                                   u^{3 \lambda_k}
                                                   \exp \left(- C_0 un\right)+ C_k \hspace{0.2 cm} \frac{(1-\log(u))^3}{n^{3\lambda_k/2}}\right)du \\
&= \bo\left( \frac{1}{n^{3\lambda_k/2}}\right)
\end{align*}
as $n \to \infty$. This implies the assertion.
\end{proof}

\begin{lem} \label{bn} As $n\to\infty$, we have
\begin{align*}
\|b_n \|_{3} \longrightarrow 0.
\end{align*}
\end{lem}

\begin{proof}
By the triangle inequality,
\begin{align*}
\|b_n \|_{3} & \leq \|\Sigma_n\|_{\text{op}} \sum_{j=1}^2
                 \left\|\sigma_n F_n^{(j)} \right\|_{3}.
\end{align*}
We have $(I_n, U, \Xi_k^{(0)}) \stackrel{d}{=} (J_n, 1-U, \Xi_k^{(1)})$ with $\Xi_k^{(0)}$ independent of $(I_n,U)$. The triangle inequality implies
\begin{align*}
\left\|\sigma_nF_n^{(2)} \right\|_{3} & \leq \frac{4}{\sqrt{n}} \sum_{k=1}^r  n^{\lambda_k} \left\|\Xi_k^{(0)}\right\|_{3} \left\|\left( \frac{I_n}{n}\right)^{\omega^k} - U^{\omega^k} \right\|_{3} \\
&= \frac{4}{\sqrt{n}} \sum_{k=1}^r \bo \left(n^{\lambda_k/2} \right) = \so(1)
\end{align*}
by Lemma (\ref{GamU}). Also, for $n \to \infty$,
\begin{align*}
\left\|\sigma_nF_n^{(1)} \right\|_{3} & \leq \frac{2}{\sqrt{n}}
                                          \left(\sum_{k=1}^r \left\|
                                          G_{k,n}(I_n) -
                                          n^{\omega^k}g_k(U)\right\|_{3}
                                          \right. \\
&+ \left.
                                          \frac{1}{\sqrt{\log(n)}}\left\| G_{r+1,n}(I_n) \right\|_{3} +\sum_{k=r+2}^{m/2-1} \left\| G_{k,n}(I_n) \right\|_{3} \right) \\
& \leq \frac{2}{\sqrt{n}} \left(\sum_{k=1}^r \frac{2}{\Gamma(1+\omega^k)}n^{\lambda_k}\left\|\left( \frac{I_n}{n}\right)^{\omega^k} - U^{\omega^k} \right\|_{3}  \right. \\
& \left. + \frac{1}{\sqrt{\log(n)}}\left\| G_{r+1,n}(I_n)
  \right\|_{3} + \sum_{k=r+2}^{m/2-1} \left\| G_{k,n}(I_n)
  \right\|_{3} \right) + \so(1)\\
&= \so(1)
\end{align*}
as before. Now, the sequence $(\|\Sigma_n\|_{\text{op}})_{n\ge 0}$ is
convergent and thus bounded, which implies the claim.
\end{proof}

Finally, we use recursion (\ref{RekN}) for $N_n$ to show that the sequence $(\|N_n\|_{3} )_{n\ge 0}$  is bounded.

\begin{lem} \label{bounded} As $n\to\infty$, we have
\begin{align*}
\|N_n\|_{3} = \bo(1).
\end{align*}
\end{lem}

\begin{proof}
 Recall that the composition vector $R_n$ takes only finitely many values, the random variables $\Xi_k$ have finite absolute moments of arbitrary order, see (\ref{mg_conv}), and $\|\Sigma_n\|_{\text{op}} \to 1$. Hence, we have $\|N_n\|_{3} < \infty$ for all $n \geq 0$.

Recursion (\ref{RekN}) implies that
\begin{align*}
\|N_n\| & \leq \mathcal{Y}^{(0)} + \mathcal{Y}^{(1)} + \|b_n\|,
\end{align*}
where $\mathcal{Y}^{(0)} := \left\|A_n^{(0)}\right\|_{\text{op}} \left\|N_{I_n}^{(0)}\right\|$ and $\mathcal{Y}^{(1)} := \left\|A_n^{(1)}\right\|_{\text{op}} \left\|N_{J_n}^{(1)}\right\|$. For all $n \geq 0$,
\begin{align}
\E\left[ \|N_n\|^3\right] &\leq \E\left[\left( \mathcal{Y}^{(0)} \right)^3 \right] + \E\left[\left( \mathcal{Y}^{(1)} \right)^3 \right] + \E\left[\|b_n\|^3\right] + 3 \E\left[\left( \mathcal{Y}^{(0)} \right)^2 \mathcal{Y}^{(1)} \right] \nonumber\\
& + 3  \E\left[\left( \mathcal{Y}^{(1)} \right)^2 \mathcal{Y}^{(0)} \right] + 3 \E\left[\left( \mathcal{Y}^{(0)} \right)^2 \|b_n\|\right] +  3 \E\left[ \mathcal{Y}^{(0)}  \|b_n\|^2\right] \nonumber\\
&+ 3 \E\left[\left( \mathcal{Y}^{(1)} \right)^2 \|b_n\|\right] +  3 \E\left[ \mathcal{Y}^{(1)}  \|b_n\|^2\right] + 6 \E\left[ \mathcal{Y}^{(0)} \mathcal{Y}^{(1)} \|b_n\|\right] .\label{langesumme}
\end{align}
Set
\begin{align*}
\beta_n &:= 1 \vee \max_{0 \leq k \leq n} \E\left[ \|N_k\|^3\right].
\end{align*}
By Lemma \ref{bn}, $\E\left[\|b_n\|^3\right] \to 0$ as $n \to \infty$. Also,
\begin{align*}
\E\left[\left( \mathcal{Y}^{(j)} \right)^3 \right] &= \E \left[ \left\|A_n^{(j)}\right\|_{\text{op}}^3 \sum_{k=0}^{n-1} \ind(I_n = k) \E\left[\|N_k\|^3 \right] \right]  \leq \E \left[ \left\|A_n^{(j)}\right\|_{\text{op}}^3 \right] \beta_{n-1}
\end{align*}
for $j=0,1$.

To bound the summand $\E\left[\left( \mathcal{Y}^{(0)} \right)^2 \mathcal{Y}^{(1)} \right]$, note that $\left\|A_n^{(0)}\right\|_{\text{op}}$ and $\left\|A_n^{(1)}\right\|_{\text{op}}$ are uniformly bounded in $n$. This implies that after conditioning on $I_n$, there is a constant $D>0$ such that
\begin{align*}
\E\left[\left( \mathcal{Y}^{(0)} \right)^2 \mathcal{Y}^{(1)} \right] & \leq D \E\left[ \sum_{k=0}^{n-1} \ind(I_n=k) \E\left[\|N_k \|^2\right] \E\left[\|N_{n-1-k} \|\right] \right]\\
& \leq D \left( \max_{0 \leq k \leq n-1} \|N_k\|_{2}^2\right) \left(  \max_{0 \leq k \leq n-1} \|N_k\|_{1}\right).
\end{align*}
Now, by construction, $\Cov(N_n)=M_m$ for all $n \geq n_0$, so
$\max_{0 \leq k \leq n-1} \|N_k\|_2^2 < K$ for some $K>0$ and hence $\E\left[\left( \mathcal{Y}^{(0)} \right)^2 \mathcal{Y}^{(1)} \right] = \bo(1)$. The same applies to $\E\left[\left( \mathcal{Y}^{(1)} \right)^2 \mathcal{Y}^{(0)} \right] $.

All other summands in (\ref{langesumme}) can be bounded using H\"older's inequality. Combining all these bounds leads to the estimate
\begin{align*}
\E\left[ \|N_n\|^3\right] &\leq \left(\E\left[ \left\|A_n^{(0)}\right\|_{\text{op}}^3 + \left\|A_n^{(1)}\right\|_{\text{op}}^3\right] + \so(1) \right) \beta_{n-1} + \bo(1).
\end{align*}
The asymptotics in Lemma \ref{Acoeff} further imply
\begin{align*}
\E\left[ \|N_n\|^3\right] &\leq \left(\E\left[ U^{3/2} + (1-U)^{3/2}\right] + \so(1) \right) \beta_{n-1} + \bo(1) = \left(\frac{4}{5} + \so(1)\right)\beta_{n-1} + \bo(1).
\end{align*}
Since $\beta_n \geq 1$, there exist $J \in \N$ and a constant $0 < E < \infty$ such that for all $n \geq J$, $\E\left[ \|N_n\|^3\right] \leq (9/10) \beta_{n-1} + E$. Induction on $n$ gives that for all $n \geq 0$, $\E\left[ \|N_n\|^3\right] \leq  \max\{\beta_{J},10 E\}$.
\end{proof}

\subsection{Proof of Proposition \ref{prop1}}\label{sec:34}
\begin{proof}[Proof of Proposition \ref{prop1}]
Proposition \ref{prop1} claims the convergence $Z_n
\stackrel{d}{\longrightarrow} \mathcal{N}$, as $n \to \infty$, where
$\mathcal{N} \sim \mathcal{N}(0,M_m)$. In order to establish this convergence, the key point is to show that
\begin{align*}
\zeta_3(N_n, \mathcal{N}) \longrightarrow 0 \hspace{0.5 cm} \text{as} \hspace{0.3 cm} n \to \infty.
\end{align*}
This is sufficient, as the difference $Z_n - N_n$ tends to
$0$ in probability and convergence in the Zolotarev metric
implies weak convergence of probability measures on $\mathbb{R}^{m-1}$. 

Recall that $N_n$ satisfies (\ref{RekN}) and that $\mathcal{N}(0,M_m)$ is
a solution to the distributional recursion
\begin{align*}
\mathcal{N} \stackrel{d}{=} \sqrt{U}  \mathcal{N}^{(0)} +
  \sqrt{1-U}\mathcal{D} \mathcal{N}^{(1)},
\end{align*}
where $\mathcal{N}^{(0)},\mathcal{N}^{(1)}$ and $U$ are independent,
$U$ is uniform on $[0,1]$ and $\mathcal{N}^{(0)}$ and $\mathcal{N}^{(1)}$ have the same distribution as  $\mathcal{N}$.

First, we use recursion (\ref{RekN}) for $N_n$ to define hybrid random
variables that link $N_n$ to $\mathcal{N}(0,M_m)$ as follows: Let
$\mathcal{N}^{(0)}$ and $\mathcal{N}^{(1)}$ be defined on the same
probability space as $(U_n)_{n \ge 1}$, independent with
distribution $\mathcal{N}(0,M_m)$ and also independent of
$(U_n)_{n \geq 1}$. We eliminate the error term in the given recursion
and set
\begin{align*}
Q_n := A_n^{(0)}\left(\ind(I_n < n_0)N_{I_n}^{(0)}+\ind(I_n \geq n_0) \mathcal{N}^{(0)}\right) + A_n^{(1)} \left(\ind(J_n < n_0)N_{J_n}^{(1)}+\ind(J_n \geq n_0) \mathcal{N}^{(1)}\right)
\end{align*}
for $n \geq 1$ with $Q_0:=N_0$.
$Q_n$ does not necessarily have covariance matrix $M_m$. However, note that $I_n/n$ converges to the uniform random variable $U$ almost surely. Together with Lemma \ref{Acoeff}, we obtain
\begin{align*}
\Cov(Q_n)  \to M_m.
\end{align*}
In order to ensure finiteness of the Zolotarev metric, the covariance
matrix of $Q_n$ has to be adjusted. Due to the convergence of the covariance matrix, $\Cov(Q_n) $ has full rank for all
$n \ge n_1$.
This implies that we can find
a deterministic sequence of matrices $(B_n)_{n\ge 0}$ with $\Cov( B_n
Q_n) = M_m$ for all $n \geq n_1$ and $B_n \to \mathrm{Id}_{m-1}$
componentwise and in operator norm as $n \to \infty$. 
We write $B_n
= \mathrm{Id}_{m-1} + K_n$ with $(K_n)_{n\ge 0}$ tending to the all zero matrix componentwise. 

Without loss of generality, we assume that $n_1 \geq n_0$ in the following. Hence, with $\mathcal{N}$ as before and $n \geq n_1$, each pair of $N_n$, $(\text{Id}_{m-1}+K_n)Q_n$ and $\mathcal{N}$ is $\zeta_3$-compatible and the triangle inequality implies
\begin{align} \label{split}
\zeta_3(N_n, \mathcal{N}) &\leq \zeta_3(N_n, (\text{Id}_{m-1}+K_n)Q_n) + \zeta_3((\text{Id}_{m-1}+K_n)Q_n,\mathcal{N}),
\end{align}
which is finite for all $n \geq n_1$.

First we show that $\zeta_3((\text{Id}_{m-1}+K_n)Q_n,\mathcal{N}) = \so(1)$ by use of an upper bound of $\zeta_3$ by the minimal $L_3$-metric  $\ell_3$.
The  minimal $L_3$-metric $\ell_3$ is given by
\begin{align}\label{ell_p}
\ell_3(X,Y):= \ell_3({\cal L}(X),{\cal L}(Y)):= \inf\{\|X'-Y'\|_3: {\cal L}(X)={\cal L}(X'), {\cal L}(Y)={\cal L}(Y')\},
\end{align}
for all  random vectors $X$, $Y$ with $\|X\|_3, \|Y\|_3<\infty$.
 For a $\zeta_3$-compatible pair $(X,Y)$, we have the inequality, see \cite[Lemma 5.7]{DrJaNe08},
\begin{align*}
\zeta_3(X,Y) \leq \left( \|X\|_{3}^2 + \|Y\|_{3}^2\right) \ell_3(X,Y).
\end{align*}
As $\sup_{n \geq 0} \|Q_n\|_{3} < \infty$ by Lemma \ref{Acoeff} and the properties of the Gaussian distribution, also $\|(\text{Id}_{m-1}+K_n)Q_n) \|_{3}$ is uniformly bounded in $n$. So there exists a finite constant $C>0$ with
\begin{align*}
\zeta_3((\text{Id}_{m-1}+K_n)Q_n,\mathcal{N}) \leq C \ell_3((\text{Id}_{m-1}+K_n)Q_n,\mathcal{N})
\end{align*}
for all $n \geq n_1$. In order to upper bound the latter $\ell_3$-distance,
note that the random vectors $\mathcal{N}$ and
$\sqrt{U}\mathcal{N}^{(0)} + \sqrt{1-U}\mathcal{D}\mathcal{N}^{(1)}$ are
identically distributed. Thus for $n \geq n_1$,
\begin{align*}
& \zeta_3((\text{Id}_{m-1}+K_n)Q_n,\mathcal{N}) \leq C \ell_3((\text{Id}_{m-1}+K_n)Q_n,\mathcal{N})
  \\
 &\leq  C \left\|\left( (\text{Id}_{m-1} + K_n)A_n^{(0)}\ind(I_n \geq n_0) - \sqrt{U}
  \mathrm{Id}_{m-1}\right) \mathcal{N}^{(0)} \right. \\
  & \left. \quad + \left( (\text{Id}_{m-1} + K_n)A_n^{(1)}\ind(J_n \geq n_0) - \sqrt{1-U}\mathcal{D}
  \right) \mathcal{N}^{(1)} \right\|_{3} \\
  & \quad + C \left\|(\text{Id}_{m-1} + K_n)A_n^{(0)}\ind(I_n < n_0)N_{I_n}^{(0)} + (\text{Id}_{m-1} + K_n)A_n^{(1)}\ind(J_n < n_0)N_{J_n}^{(1)} \right\|_3 \\
 &\leq C \left( \left\| (\text{Id}_{m-1} + K_n)A_n^{(0)} - \sqrt{U}
  \text{Id}_{m-1}\right\|_{3} \left\|\mathcal{N}^{(0)} \right\|_{3} \right. \\
  & \quad +\left. \left\| (\text{Id}_{m-1} + K_n)A_n^{(1)} - \sqrt{1-U}\mathcal{D}
  \right\|_{3}\left\| \mathcal{N}^{(1)} \right\|_{3}
  \right) \\ 
  & \quad +  C \left\|(\text{Id}_{m-1} + K_n)A_n^{(0)}\ind(I_n < n_0)N_{I_n}^{(0)} + (\text{Id}_{m-1} + K_n)A_n^{(1)}\ind(J_n < n_0)N_{J_n}^{(1)} \right\|_3 \stackrel{n \to \infty}{\longrightarrow} 0.
\end{align*}

To bound the first summand in (\ref{split}), we split $N_n$ into two parts and
consider the vector
\begin{align*}
\Phi_n:= A_n^{(0)} N_{I_n}^{(0)} + A_n^{(1)} N_{J_n}^{(1)}, \quad n \geq 1,
\end{align*}
with $\Phi_0:= N_0$ such that $N_n = \Phi_n + b_n$. An application of Lemma \ref{zolosum} to
the sums $N_n = \Phi_n + b_n$ and (Id$_{m-1}+K_n$)$Q_n = Q_n + K_nQ_n$
gives for $n \geq n_1$ that
\begin{align*}
\zeta_3(N_n, (\text{Id}_{m-1}+K_n)Q_n) & \leq \zeta_3(\Phi_n, Q_n) +
                                     \|\Phi_n\|_{3}^2\|b_n\|_{3} +
                                     \frac{1}{2}\|\Phi_n\|_{3}\|b_n\|_{3}^2
+ \frac{1}{2}\|b_n\|_{3}^3\\
&~\qquad + \left(\|K_n\|_{\text{op}} + \frac{1}{2}
  \|K_n\|_{\text{op}}^2 + \frac{1}{2}
  \|K_n\|_{\text{op}}^3\right) \|Q_n\|_{3}^3.
\end{align*}
By construction, $\|K_n\|_{\text{op}} \to 0$ and by Lemma \ref{bn},
$\|b_n\|_{3} \to 0$. Also, by Lemma \ref{bounded}, $\sup_{n \geq
  0}\|\Phi_n\|_{3} < \infty$ and $\sup_{n \geq 0}\|Q_n\|_{3} <
\infty$, this yields for $n \geq n_1$ that
\begin{align*}
\zeta_3(N_n, (\text{Id}_{m-1}+K_n)Q_n) \leq \zeta_3(\Phi_n, Q_n) + \so(1).
\end{align*}
The previous estimates and (\ref{split}) imply that for $n \geq n_1$,
\begin{align}
 \zeta_3(N_n, \mathcal{N}) & \leq \zeta_3 \left(\Phi_n, Q_n\right) + \so(1).
\end{align}

Let $\Delta(n):= \zeta_3(N_n, \mathcal{N})$, which is finite for $n \geq n_1$. Note that $\zeta_3\left(\Phi_n, Q_n\right)$ is finite for $n \geq 0$. In the expectations defining the Zolotarev distance, we condition on the
value of $I_n$. With $(N_0^{[0]}, \ldots,
N_{n-1}^{[0]})$,$(N_0^{[1]}, \ldots, N_{n-1}^{[1]})$ i.i.d.~with
distribution $\mathcal{L}(N_0, \ldots, N_{n-1})$ we make use of
independence and the fact that $\zeta_3$ is $(3,+)$-ideal and satisfies (\ref{op_raus}) to get, again for $n \geq n_1$,
\begin{align*}
\zeta_3 \left(\Phi_n, Q_n\right) & \leq  \frac{1}{n} \sum_{k=0}^{n_0-1}\zeta_3\left(
  \Sigma_n\sigma_n
  \sigma_{n-1-k}\mathcal{D}\Sigma_{n-1-k}^{-1}N_{n-1-k}^{[1]},\Sigma_n\sigma_n
  \sigma_{n-1-k}\mathcal{D}\Sigma_{n-1-k}^{-1}\mathcal{N}^{(1)}  \right) \\
  & +  \frac{1}{n} \sum_{k=n-n_0}^{n-1}\zeta_3\left(
  \Sigma_n\sigma_n
  \sigma_{k}\Sigma_{k}^{-1}N_{k}^{[0]},\Sigma_n\sigma_n
  \sigma_{k}\Sigma_{k}^{-1}\mathcal{N}^{(0)}  \right) \\
  & +  \frac{1}{n} \sum_{k=n_0}^{n-n_0-1} \zeta_3\left(\Sigma_n\sigma_n\sigma_{k}^{-1}\Sigma_k^{-1}N_k^{[0]} +
  \Sigma_n\sigma_n
  \sigma_{n-1-k}\mathcal{D}\Sigma_{n-1-k}^{-1}N_{n-1-k}^{[1]},
  \right.\\
& \left. \quad\quad\quad\quad\quad\; \Sigma_n\sigma_n\sigma_k^{-1}\Sigma_k^{-1}\mathcal{N}^{(0)} +
 \Sigma_n \sigma_n\sigma_{n-1-k}^{-1} \mathcal{D}\Sigma_{n-1-k}^{-1}\mathcal{N}^{(1)}
  \right)\\
  & \leq \frac{2}{n} \sum_{k=n-n_0}^{n-1} \|\sigma_n \sigma_k^{-1}\|_{\text{op}}^3 \|\Sigma_n\|_{\text{op}}^3
  \|\Sigma_k^{-1}\|_{\text{op}}^3 \zeta_3\left(N_k^{[0]},
  \mathcal{N}^{(0)} \right)\\
  & + \frac{2}{n} \sum_{k=n_0}^{n-n_0} \|\sigma_n \sigma_k^{-1}\|_{\text{op}}^3 \|\Sigma_n\|_{\text{op}}^3
  \|\Sigma_k^{-1}\|_{\text{op}}^3 \zeta_3\left(N_k^{[0]},
  \mathcal{N}^{(0)} \right)\\
& = \frac{2}{n} \sum_{k=n_0}^{n-1} \|\sigma_n \sigma_k^{-1}\|_{\text{op}}^3 \|\Sigma_n\|_{\text{op}}^3
  \|\Sigma_k^{-1}\|_{\text{op}}^3 \zeta_3\left(N_k^{[0]},
  \mathcal{N}^{(0)} \right) .
\end{align*}
Note that
$\|\sigma_n \sigma_{I_n}^{-1}\|_{\text{op}}^3 =
\left(\frac{I_n}{n}\right)^{3/2}$ in both cases $6 \mid m$ and $6 \nmid m$. Hence, for $6 \mid m$ and $n \geq n_1$,
\begin{align*}
\Delta(n) \leq 2 \E\left[\left(\frac{I_n}{n}\right)^{3/2} \|\Sigma_n\|_{\text{op}}^3
  \|\Sigma_{I_n}^{-1}\|_{\text{op}}^3 \Delta(I_n)\ind(I_n \geq n_0)  \right] + \so(1).
\end{align*}
Now a standard argument shows that $\zeta_3(N_n, \mathcal{N})
\to 0$ as $n \to \infty$, see \cite{ne14}, for example.
\end{proof}

\subsection*{Acknowledgement}  We thank Henning Sulzbach  for helpful comments.

\end{document}